\documentclass[11pt,noamsfonts]{amsart}
\usepackage{unicode-math, aliascnt}
\usepackage{microtype}
\usepackage[colorlinks=true, linkcolor=blue, citecolor=blue, urlcolor=blue]{hyperref}
\hypersetup{
  pdftitle={Classification of strictly resistance nonnegative graphs},
  pdfauthor={Hailey Jay Garcia},
  pdfsubject={Combinatorics: resistance curvature, toughness, and polyhedral combinatorics of graphs},
  pdfkeywords={resistance curvature, toughness, spanning-tree polytope,
    effective resistance, fractional matching, implicit equality, bipartite graph},
}
\usepackage[capitalise,nameinlink,sort]{cleveref}
\usepackage{enumitem}
\AtBeginDocument{}

\crefname{equation}{Equation}{Equations}
\Crefname{equation}{Equation}{Equations}
\creflabelformat{equation}{#2#1#3}

\newcommand{\R}{\mathbb{R}}
\newcommand{\cT}{\mathcal{T}}
\newcommand{\Th}{\Theta}
\newcommand{\FM}{\mathit{FM}}
\newcommand{\iv}{\mathbf{1}}
\newcommand{\RN}{\textup{RN\@}}
\newcommand{\RP}{\textup{RP\@}}

\newcommand{\newsharedtheorem}[3]{%
  \newaliascnt{#1}{theorem}%
  \newtheorem{#1}[#1]{#2}%
  \aliascntresetthe{#1}%
  \crefname{#1}{#2}{#3}%
  \Crefname{#1}{#2}{#3}%
}

\newtheorem{theorem}{Theorem}[section]
\numberwithin{equation}{section}

\renewcommand{\themainthm}{\Alph{mainthm}}
\crefname{mainthm}{Theorem}{Theorems}
\Crefname{mainthm}{Theorem}{Theorems}
\newsharedtheorem{lemma}{Lemma}{Lemmas}
\newsharedtheorem{question}{Question}{Questions}
\theoremstyle{definition}
\newtheorem*{remark}{Remark}
\newsharedtheorem{example}{Example}{Examples}

\usepackage{thm-restate}

\title[Strictly resistance nonnegative graphs]{Classification of strictly resistance nonnegative graphs}

\author[H. J. Garcia]{Hailey Jay Garcia}
\address{Louisiana State University}
\email{jgarc86@lsu.edu}

\date{\today}


\begin{document}

\begin{abstract}
We say that a graph is resistance nonnegative or \RN{} if it admits a positive edge-weight that yields nonnegative resistance curvature in the sense of Devriendt and Lambiotte.
Analogously, a graph may be resistance positive or \RP{}; we say a graph is strictly \RN{} if it is \RN{} but not \RP{}.

In this paper, we show that every $2$-connected strictly \RN{} graph is bipartite with parts whose sizes differ by one, demonstrating that there are no $1$-tough strictly \RN{} graphs.
As a consequence, we prove that every $1$-tough \RN{} graph is also \RP{}.
Lastly, we quantify the exact toughness values an \RN{} graph can attain below $1$.
\end{abstract}

\maketitle

\section{Introduction}
\label{sec:intro}

For a finite, simple, connected graph, the shortest path length between two vertices is an elegant notion of distance, but has drawbacks.
For example, consider a graph where the edge $uw$ is the only path between the vertices $u$ and $w$.
Cutting that edge would separate those vertices and the graph.
However, in a graph with many paths between $u$ and $w$, the edge $uw$ would, in some sense, be ``less important'' in connecting $u$ and $w$.
In either case, the distance between them is $1$ and tells us nothing of the global structure of the graph.
Another way to measure how well the edge $uw$ connects its two endpoints comes from electrical network theory~\cite{kirchhoff1847,doylesnell1984}.
Given a positive edge-weight, or conductance, on each edge, one views the graph as an electrical network.
The \emph{effective resistance} between two vertices $u$ and $w$ is the potential difference between them when one unit of current enters the network at $u$ and leaves it at $w$.
When $uw$ is an edge, we take its endpoints as the terminals and call the resulting quantity the effective resistance \emph{of the edge} $uw$.
This is a global notion, which depends on all paths connecting $u$ and $w$ rather than just $uw$.

This resistance gives rise to a discrete geometry and lets one define the \emph{resistance curvature} at each vertex~\cite{devriendtlambiotte2022}.
A graph is \emph{resistance nonnegative}, or \RN{}, if some positive edge-weighting makes the resistance curvature nonnegative at every vertex, and \emph{resistance positive}, or \RP{}, if some positive edge-weighting makes it positive at every vertex.
That \RP{} graphs are also \RN{} is immediate.
Graphs that are \RN{} but not \RP{} are called \emph{strictly resistance nonnegative}.

Determining which graphs are \RN{}, strictly so, or \RP{} is related to the \emph{toughness} $\tau$ of a graph~\cite{chvatal1973}.
By a theorem of Fiedler~\cite[Theorem~3.4.18]{fiedler2011} as applied by Devriendt~\cite[Theorem~1.3]{devriendt2026}, all \RP{} graphs are $1$-tough.
A bipartite \RN{} graph has parts whose sizes differ by at most one; when they do differ, the graph is not $1$-tough and hence strictly \RN{}.
The converse has not been addressed.

To fully investigate this phenomenon, we require more machinery.
Many proofs in this paper rely on the \emph{tree-double-matching polytope} $\Th$ associated to $G$, the intersection of its spanning-tree polytope with the double of its fractional matching polytope.
Each vertex $v$ carries the inequality $x(E(v)) \le 2$ on $\Th$, and we call $v$ \emph{slackless} if that inequality is tight at every point of $\Th$, writing $U$ for the set of slackless vertices.
A strict degree inequality at $v$ is positive curvature at $v$, so no edge-weight making every curvature nonnegative can give a slackless vertex positive curvature; hence, a graph with $U\neq \emptyset$ is not \RP{}.
We formalise this observation in \Cref{thm:implicit} and use the set $U$ to show the following.

\begin{restatable}{mainthm}{thmchar}
\label{thm:char}
For $2$-connected $G$ on $n\ge 3$ vertices, the following are equivalent:
\begin{enumerate}[label=(\alph*), ref=\themainthm(\alph*)]
  \item\label[mainthm]{thm:char-strict} $G$ is strictly \RN{};
  \item\label[mainthm]{thm:char-slackless} $G$ is \RN{} and has a slackless vertex;
  \item\label[mainthm]{thm:char-bip} $G$ is \RN{} and bipartite with parts of sizes $k$ and $k+1$ for some $k \ge 2$.
\end{enumerate}
In that case $n = 2k+1$ is odd, $U$ is the smaller part, and
\[
  \tau(G) = \frac{n-1}{n+1} = \frac{k}{k+1} < 1,
\]
and $U$ is the only cut-set attaining $\tau(G)$.
\end{restatable}

The only \RN{} graphs on $n \ge 3$ vertices that are not $2$-connected are the paths $P_n$~\cite[Proposition~3.7]{devriendt2026}, which have toughness $\tau=1/2$.
Hence, strictly \RN{} graphs have a specific structure, and ruling out this structure is all that is required to identify an \RN{} graph as \RP{}.
The \RN{} graphs with toughness below $1$ are therefore exactly the strictly \RN{} graphs; with the theorem of Fiedler above, this gives the following.

\begin{restatable}{mainthm}{thmband}
\label{thm:band}
For connected $G$ on $n \ge 3$ vertices, $G$ is \RP{} if and only if $G$ is \RN{} and $1$-tough.
\end{restatable}

Neither hypothesis is redundant.
Every \RP{} graph is $1$-tough but not conversely: Agrahari et al.~\cite[Theorem~4]{agraharietal2026} construct, for every $n \ge 11$, a $1$-tough graph on $n$ vertices that is not \RN{}.

In a recent preprint, Wang~\cite[Corollary~5]{wang2026} showed that every $10$-tough graph is \RP{} using fractional Hamiltonicity, settling a conjecture by Devriendt~\cite[Conjecture~6.4]{devriendt2026}.
In another recent preprint, Huang~\cite[Theorem~1.4]{huang2026} claims a result that, combined with~\cite[Theorem~1]{wang2026}, would reduce the bound to $5$.
By \Cref{thm:band}, for $t\ge 1$, every $t$-tough graph is \RN{} if and only if every $t$-tough graph is \RP{}.
Wang~\cite[Theorem~6]{wang2026} constructs, for every $\varepsilon > 0$, a graph that is not \RN{} with $\tau > 3/2 - \varepsilon$.
Write $t^*_{\RN{}} = \inf\{t : \text{every } t\text{-tough graph is \RN{}}\}$ and $t^*_{\RP{}}$ analogously.
Wang's results~\cite{wang2026} give $3/2\le t^*_{\RN{}}\le t^*_{\RP{}}\le 10$.
Since $t^*_{\RN{}}\ge 3/2>1$, the two infima are taken over the same set, and \Cref{thm:band} strengthens the central inequality to $t^*_{\RN{}}=t^*_{\RP{}}$.
Combining these results on $1$-tough \RN{} graphs with the paths $P_n$ and the family $K_{k,k+1}$ constructed in~\cite[Example~3.4]{devriendt2026}, we determine the exact toughness that \RN{} graphs attain below $1$.

\begin{restatable}{mainthm}{thmtough}
\label{thm:tough}
The toughness values below $1$ attained by connected \RN{} graphs are exactly $k/(k+1)$ for $k\ge 1$.
In particular, a $2$-connected \RN{} graph on $n$ vertices has either $\tau \ge 1$, or $n$ odd and $\tau = (n-1)/(n+1)$; and for every $k \ge 2$ the complete bipartite graph $K_{k,k+1}$ is $2$-connected and strictly \RN{} with $\tau = k/(k+1)$.
\end{restatable}

\Cref{sec:prelim} begins our discussion of resistance positivity and nonnegativity by collecting the polytope descriptions into \Cref{thm:implicit}.
Next, we will demonstrate certain cut-set inequalities in \Cref{sec:cutsets} and recover necessary conditions for a bipartite graph to be \RN{} or \RP{} as Devriendt did in~\cite[Proposition~3.5]{devriendt2026}.
\Cref{sec:slack} uses our definition of slackless to prove \Cref{thm:extremal}, and then a brief detour into linear programming duality yields \Cref{lem:sat}, which shows that a nonempty slackless set in a $2$-connected \RN{} graph is an independent vertex cover.
\Cref{sec:cons} includes our main results \Cref{thm:char,thm:band,thm:tough} to characterise all strictly \RN{} graphs, the result that \RN{} and $1$-tough is equivalent to \RP{}, and the toughness they are permitted.
\Cref{sec:open} closes with two open questions: which bipartite graphs with parts of sizes $k$ and $k+1$ are \RN{}, and the value of the common toughness threshold $t^*_{\RN{}}=t^*_{\RP{}}$.

\section{Preliminaries}
\label{sec:prelim}

Throughout, $G$ is a simple, finite connected graph on $n$ vertices, $V = V(G)$ and $E = E(G)$ are its vertex and edge sets, and $\cT = \cT(G)$ is its set of spanning trees.
When the graph is clear from context, we drop the argument and write $\cT$, $V$, $E$, and so on.
We take $n\ge 3$ and treat the graphs $K_1$ and $K_2$ separately in \Cref{ex:k1k2}.
For a graph $H$ we write $c(H)$ for its number of components.

For $S \subseteq V$ we write $E[S]$ for the set of edges with both endpoints in $S$ and $E(v)$ for the set of edges meeting $v$.
We write $G[S]$ for the subgraph of $G$ induced by $S$.
For disjoint $A, B \subseteq V$ we write $E(A,B)$ for the set of edges with one endpoint in $A$ and the other in $B$.
For $F \subseteq E$ and $x \in \R^{E}$ we write $x(e)$ for the value of $x$ at the edge $e$.
We extend to subsets naturally by $x(F) = \sum_{e \in F} x(e)$.

Let $\alpha:E\to \R_+$ be a positive edge-weight on $G$.
For an edge $e\in E$, the \emph{effective resistance} of $e$ with respect to $\alpha$ is defined as
\[ \omega_e(\alpha)=\alpha(e)^{-1}\frac{\sum_{T\in \cT,T\owns e}\prod_{f\in T}\alpha(f)}{\sum_{T\in \cT}\prod_{f\in T}\alpha(f)}. \]
From effective resistance, \emph{resistance curvature} in the sense of~\cite{devriendtlambiotte2022} is defined as
\[ p_v(\alpha)=1-\frac{1}{2}\sum_{e\in E(v)}\alpha(e)\omega_e(\alpha). \]
We say that a graph is \emph{resistance nonnegative} or \RN{} if there is an $\alpha$ with $p_v(\alpha)\ge 0$ for all $v\in V$.
Analogously, a graph is \emph{resistance positive} or \RP{} if there is an $\alpha$ with $p_v(\alpha)> 0$ for all $v\in V$.
A graph is \emph{strictly} \RN{} if it is \RN{} but not \RP{}.

The \emph{toughness} $\tau(G)$ of a graph $G$ is
\[
  \tau(G) = \min \left\{ \frac{|S|}{c(G-S)} : S \subseteq V,\ c(G-S) \ge 2 \right\},
\]
the minimum taken over all \emph{cut-sets} of $G$, that is, over all $S \subseteq V$ with $c(G-S) \ge 2$.
We adopt the convention that $\tau(K_n)=\infty$.
A cut-set $S$ \emph{attains} $\tau(G)$ if $|S|/c(G-S) = \tau(G)$.
Following~\cite{chvatal1973}, we say that $G$ is \emph{$1$-tough} if $\tau(G) \ge 1$.

\begin{example}\label{ex:k1k2}
The only connected graphs with fewer than three vertices are $K_1$ and $K_2$, both \RP{}.
Because $K_1$ has no edges, $p_v = 1$ trivially.
Since $K_2$ has one spanning tree, $p_v = 1-1/2 = 1/2$ for any vertex $v$.
Both are complete, so $\tau=\infty$.
\end{example}

We define several polytopes in $\R^E$ associated to $G$.
The \emph{spanning-tree polytope} of $G$ is
\[
  P = \operatorname{conv}\{\iv_T : T \in \cT\} \subseteq \R^{E},
\]
the convex hull of the indicator vectors $\iv_T$ over all spanning trees $T$; it is the graphic case of the \emph{matroid base polytope}~\cite{edmonds1971}.
Let $\FM$ be the \emph{fractional matching polytope} of $G$~\cite[Section~7.5]{lovaszplummer1986}.
Its double is
\begin{multline}\label{eq:2FM}
  2\FM = \{x \in \R^{E} : x(e) \ge 0 \text{ for all } e \in E,\\
  2 - x(E(v)) \ge 0 \text{ for all } v \in V\},
\end{multline}
whose vertices are the basic $2$-matchings of $G$~\cite[Theorem~7.5.1]{lovaszplummer1986}.
We call $2 - x(E(v)) \ge 0$ the \emph{degree} inequality at $v$.
The polytope $\FM$ is a relaxation of the matching polytope $M$ used in~\cite{devriendt2026} and defined in~\cite{edmonds1965}; see the following remark.
In contrast to $P$, the polytope $2\FM$ is full-dimensional; for sufficiently small $\varepsilon > 0$, the point $\varepsilon \iv_E$ satisfies all defining inequalities strictly.
The \emph{tree-double-matching polytope} of $G$ is the intersection $\Th = P \cap 2\FM$~\cite[Definition~4.4]{devriendt2026}.

\begin{remark}
Devriendt's double-matching polytope $2M$~\cite[Section~4]{devriendt2026} is twice the matching polytope $M$, the convex hull of the indicator vectors of matchings~\cite[Section~7.2]{lovaszplummer1986}.
The polytope $2\FM$ contains it.
By Edmonds' matching polytope theorem~\cite{edmonds1965}, $2M$ is defined by nonnegativity, the degree inequalities $x(E(v)) \le 2$, and in addition the \emph{odd-set} inequalities $x(E[S]) \le |S| - 1$ for $S \subseteq V$ with $|S|\ge 3$ odd.
The inequalities are the doubled form of Edmonds' odd-set inequalities $x(E[S]) \le (|S|-1)/2$.
Each such inequality is a rank inequality, or with $S=V$ the inequality form of the count equality, so they are redundant over the spanning-tree polytope and $P\cap 2M = P\cap 2\FM$.

The relaxation matters in the case of the relative interior.
Since both $M$ and $\FM$ are full-dimensional, a nontrivial defining inequality is strict in its relative interior.
When $n$ is odd, the odd-set inequality at $S = V$ holds with equality on all of $P$, so $P \cap 2M^{\circ} = \emptyset$.
Take, for example, the $3$-cycle $K_3$.
Since each vertex has degree $2$ in exactly one of the three spanning trees and degree $1$ in the other two, every $x\in P^{\circ}$ satisfies $x(E(v))<2$.
Thus, $P^{\circ}\cap 2\FM^{\circ}=P^{\circ}\neq \emptyset$.
Using $2\FM$ instead allows \RN{} and \RP{} to take the parallel forms of \Cref{lem:polytope}.
\end{remark}

For $X \subseteq \R^{E}$ we write $\operatorname{aff} X$ for its \emph{affine hull}, the smallest affine subspace containing $X$.
We write $P^{\circ}$ and $2\FM^{\circ}$ for the relative interiors of $P$ and $2\FM$, their interiors within $\operatorname{aff} P$ and $\operatorname{aff} 2\FM$.

\begin{lemma}
\label{lem:tree}
Let $G$ be connected.
\begin{enumerate}[label=(\alph*), ref=\thelemma(\alph*)]
  \item\label[lemma]{lem:tree-desc} The spanning-tree polytope $P$ is the set of points $x \in \R^{E}$ satisfying the \emph{count} equality $x(E) = n-1$, the \emph{edge} inequalities $x(e) \ge 0$ for $e \in E$, and the \emph{rank} inequalities $x(E[S]) \le |S|-1$ for $S \subseteq V$ with $2 \le |S| \le n-1$.
  \item\label[lemma]{lem:tree-interior} A point $x$ lies in $P^{\circ}$ if and only if $x = \sum_{T \in \cT} \lambda_T \iv_T$ for some $\lambda_T > 0$ with $\sum_{T} \lambda_T = 1$.
  In particular, if $x \in P^{\circ}$ and $y \in P$, then $(x+y)/2 \in P^{\circ}$.
  \item\label[lemma]{lem:tree-strict} If $x \in P$ satisfies every edge and rank inequality strictly, then $x \in P^{\circ}$.
  \item\label[lemma]{lem:tree-2conn} If $G$ is $2$-connected, then every $x \in P^{\circ}$ satisfies every edge and rank inequality strictly.
\end{enumerate}
\end{lemma}

\begin{proof}
Part (a) is Edmonds' description~\cite{edmonds1971,lovaszplummer1986} of the spanning-tree polytope.
That description uses the count equality, the edge inequalities, and $x(A) \le r(A)$ for every $A \subseteq E$, where $r(A)$ is the number of edges in a largest forest contained in $A$.
Since $r(E[S]) \le |S|-1$, those inequalities imply our rank inequalities.
Conversely, let $S_1, \ldots, S_m$ be the vertex sets of the components of the spanning subgraph of $G$ with edge set $A$, so that $r(A) = n-m$.
Every edge of $A$ lies in some $E[S_i]$, so as $x \ge 0$,
\[
  x(A) \le \sum_{i} x(E[S_i]) \le \sum_{i} (|S_i|-1) = n-m,
\]
bounding each term by $0$ if $|S_i| = 1$, by the count equality if $S_i = V$, and by a rank inequality otherwise.

For (b), let $b$ be the average of the $\iv_T$.
Suppose $x \in P^{\circ}$.
Since $b \in P$, for small $\varepsilon > 0$ the point $y = x + \varepsilon(x-b)$ lies in $P$, and $x = (y + \varepsilon b)/(1+\varepsilon)$.
Writing $y$ as a convex combination of the $\iv_T$ and $b$ with every weight $1/|\cT|$ puts a positive weight on every $\iv_T$.
Conversely, suppose $x = \sum_T \lambda_T \iv_T$ with every $\lambda_T > 0$.
Let $L = \{\sum_T \mu_T \iv_T : \sum_T \mu_T = 0\}$ be the linear subspace parallel to $\operatorname{aff} P$, so that $x \in P^{\circ}$ exactly when $x + z \in P$ for every sufficiently small $z \in L$.
The linear map $\mu \mapsto \sum_T \mu_T \iv_T$ from $\{\mu \in \R^{\cT} : \sum_T \mu_T = 0\}$ onto $L$ is surjective, so it has a linear right inverse.
Hence, every sufficiently small $z \in L$ is $\sum_T \mu_T \iv_T$ with $\sum_T \mu_T = 0$ and every $|\mu_T| < \lambda_T$, and $x + z = \sum_T (\lambda_T + \mu_T)\iv_T$ is a convex combination of the $\iv_T$, so it lies in $P$.
For the final claim, averaging the weights of $x$ and $y$ gives weights for $(x+y)/2$ that are all positive.

For (c), let $H = \{x \in \R^{E} : x(E) = n-1\}$, so $P \subseteq H$ by the count equality.
The edge and rank inequalities are finitely many and strict at $x$, so they remain strict on a neighbourhood $N$ of $x$ in $H$, and $N \subseteq P$ by (a).
Then $H = \operatorname{aff} N \subseteq \operatorname{aff} P \subseteq H$, so $N$ is a neighbourhood of $x$ in $\operatorname{aff} P$ contained in $P$, and $x \in P^{\circ}$.

For (d), write $x = \sum_T \lambda_T \iv_T$ with every $\lambda_T > 0$ by (b).
Every edge $e$ lies in some spanning tree $T$, so $x(e) \ge \lambda_T > 0$.
Now fix $S$ with $2 \le |S| \le n-1$.
Every $T$ has $\iv_T(E[S]) \le |S|-1$, so it suffices to find one $T$ for which $T \cap E[S]$ is not a spanning tree of $G[S]$.
Since $G$ is connected and $S \neq V$, some $w \notin S$ has a neighbour $s_1 \in S$.
Because $G$ is $2$-connected, $G - s_1$ is connected; let $R$ be a shortest path in $G - s_1$ from $w$ to $S \setminus \{s_1\}$, ending at $s_2$.
Then $Q = s_1 w R$ is a path from $s_1$ to $s_2$ with no inner vertex in $S$, and so no edge in $E[S]$.
Extend $Q$ to a spanning tree $T$.
If $T \cap E[S]$ were a spanning tree of $G[S]$, it would contain a path from $s_1$ to $s_2$ edge-disjoint from $Q$, and $T$ would contain a cycle.
\end{proof}

By \Cref{lem:tree-desc} and \Cref{eq:2FM}, $\Th = P \cap 2\FM$ is cut out by the count equality and the edge and rank inequalities of $P$, together with the degree inequalities of $2\FM$; we call these the \emph{defining inequalities} of $\Th$.

Devriendt~\cite{devriendt2026} described the graphs with nonnegative or positive resistance curvature in terms of polytopes.
For a positive edge-weight $\alpha$, put
\[x(e)=\alpha(e)\omega_e(\alpha)=\Pr[e\in T_\alpha],\]
that is, the probability that $e$ lies in a random $\alpha$-weighted spanning tree~\cite{kirchhoff1847}, where $T_\alpha$ is the random spanning tree with $\Pr[T_\alpha=T]\propto\prod_{f\in T}\alpha(f)$.
Then $x\in P^{\circ}$ by \Cref{lem:tree-interior}, taking every weight $\lambda_T$ proportional to $\prod_{f\in T}\alpha(f)$.
Foster's theorem~\cite{foster1949} gives $x(E)=n-1$, and $p_v(\alpha)=(2 - x(E(v)))/2$.
We modify Devriendt's construction to obtain the following.

\begin{lemma}
\label{lem:polytope}
A graph $G$ is \RN{} if and only if $P^{\circ} \cap 2\FM \neq \emptyset$, and \RP{} if and only if $P^{\circ} \cap 2\FM^{\circ} \neq \emptyset$.
\end{lemma}

\begin{proof}
Let $\mathcal{H}_v = \{x\in \R^E : x(E(v))\le 2\}$ and $\mathcal{H}_v^\circ = \{x\in \R^E : x(E(v))< 2\}$.
By~\cite[Corollary~3.9]{devriendt2026}, a graph $G$ is \RN{} (respectively \RP{}) if and only if $P^\circ$ and the halfspaces $\mathcal{H}_v$ (respectively open halfspaces $\mathcal{H}_v^\circ$) have a common intersection point.
Because the edge inequalities of $2\FM$ are redundant over $P$, the intersection of $P^{\circ}$ with the halfspaces $\mathcal{H}_v$ is exactly the set $P^{\circ}\cap 2\FM$.

Every edge lies in some spanning tree, so by \Cref{lem:tree-interior} the edge inequalities of $P^{\circ}$ are strict.
Since $2\FM$ is full-dimensional, the edge inequalities of $2\FM^{\circ}$ are also strict.
Inheriting the strict edge inequalities from $P^{\circ}$ and the strict degree inequalities from $\mathcal{H}_v^{\circ}$ yields precisely $P^{\circ}\cap 2\FM^{\circ}$.
\end{proof}

The \RN{} half is~\cite[Theorem~4.1]{devriendt2026}, stated there for $2M$; the two forms agree by the above remark.

A defining inequality of $\Th$ is an \emph{implicit equality} if it holds with equality at every point of $\Th$, following~\cite{schrijver1986}.
We say a vertex $v$ is \emph{slackless} if the degree inequality at $v$ is an implicit equality, that is, if $x(E(v)) = 2$ for every $x \in \Th$. 
We write $U = \{v \in V : x(E(v)) = 2 \text{ for all } x \in \Th\}$ for the \emph{slackless set} of $G$.
If $\Th=\emptyset$, we set $U=V$.
A strict degree inequality is a property of a single point of $\Th$, while membership in $U$ is a property of all of $\Th$ and records which of the degree inequalities become implicit equalities.

\begin{theorem}
\label{thm:implicit}
Let $G$ be a graph with $\Th \neq \emptyset$.
Then
\begin{enumerate}[label=(\alph*), ref=\thetheorem(\alph*)]
  \item\label[theorem]{thm:implicit-rn} $G$ is $2$-connected and \RN{} if and only if none of the edge and rank inequalities of \Cref{lem:tree-desc} is an implicit equality;
  \item\label[theorem]{thm:implicit-rp} $G$ is \RP{} if and only if $G$ is \RN{} and $U = \emptyset$.
\end{enumerate}

In particular, a $2$-connected \RN{} graph is strictly \RN{} if and only if $U \neq \emptyset$.
\end{theorem}
\begin{proof}
\begin{enumerate}[label=(\alph*)]
  \item
    For the forward direction, let $G$ be $2$-connected and \RN{} and, by \Cref{lem:polytope}, pick $x \in P^{\circ} \cap 2\FM$.
    By \Cref{lem:tree-2conn}, $x(e) > 0$ for every $e$ and $x(E[S]) < |S|-1$ for every $S$ with $2\le |S| \le n-1$.
    Hence, every edge and rank inequality holds strictly at $x$ and so none is an implicit equality.

    Conversely, suppose no edge or rank inequality is an implicit equality.
    For each edge inequality, select a point $x^e\in \Th$ satisfying the strict inequality $x^e(e)>0$.
    Similarly, for each rank inequality, select a point $x^S\in \Th$ satisfying $x^S(E[S])<|S|-1$.
    There are finitely many such inequalities, so let $x$ be the average of the points $x^e$ and $x^S$.
    Each $x^e$ and $x^S$ satisfies every inequality, and each inequality is strict at at least one such point; hence, $x$ satisfies each inequality strictly and lies in $\Th$ by convexity.
    Therefore, $x\in P^{\circ} \cap 2\FM$ by \Cref{lem:tree-strict} and $G$ is \RN{} by \Cref{lem:polytope}.

    Suppose $G$ were not $2$-connected.
    Then $G$ has a cut-vertex and hence at least two blocks~\cite[Section~3.1]{diestel2017}, so any block $B$ satisfies $2\le |B| \le n-1$.
    Pick any spanning tree $T$ of $G$.
    Then the forest $T\cap E[B]$ is connected; a $T$-path between two vertices of $B$ that left $B$ would leave and re-enter through the same cut-vertex, which a path cannot do.
    Hence, $T\cap E[B]$ is a spanning tree of $G[B]$ and $\iv_T(E[B])=|B|-1$ by~\cite[Corollary~1.5.3]{diestel2017}.
    Since the rank inequality corresponding to $B$ holds with equality at every vertex of $P$, it is implicit on all of $P$ and hence on $\Th$, a contradiction.
    Therefore, $G$ is $2$-connected.

  \item
    For the forward direction, let $G$ be \RP{} and so \RN{}.
    Then \Cref{lem:polytope} gives $x\in P^{\circ} \cap 2\FM^{\circ}$, so $x(E(v))<2$ for all $v\in V$; hence, $U=\emptyset$.

    Conversely, let $G$ be \RN{} with $U = \emptyset$.
    For each $v \in V$, choose $x^v \in \Th$ with $x^v(E(v)) < 2$, possible exactly because $v \notin U$.
    Let $x^0$ be the average of the points $x^v$.
    Each $x^v$ satisfies all degree inequalities and at least one strictly, so by convexity $x^0\in \Th$ and $x^0(E(v))<2$ for all $v\in V$.
    Since $G$ is \RN{}, \Cref{lem:polytope} lets us pick $x^1\in P^{\circ} \cap 2\FM$.
    Set $x=(x^0+x^1)/2$.
    Then $x \in P^{\circ}$ by \Cref{lem:tree-interior}, and $x(E(v))<2$ at every vertex because $x^0(E(v))<2$ and $x^1 \in 2\FM$.
    Therefore, $x\in P^{\circ} \cap 2\FM^{\circ}$, so $G$ is \RP{} by \Cref{lem:polytope}.
\end{enumerate}
\end{proof}

\section{Cut-sets}
\label{sec:cutsets}

We show in \Cref{lem:cutset} that a cut-set $S$ with $c(G-S) > |S|$ obstructs resistance curvature.
In particular, every \RP{} graph is $1$-tough.
Applying \Cref{lem:cutset} to the smaller part of a bipartite graph recovers a result of Devriendt in \Cref{lem:bip}.

\begin{lemma}
\label{lem:cutset}
Let $G$ be connected, $S \subseteq V$, and $c = c(G-S) \ge 2$.
Then
\[
  \sum_{v \in S} x(E(v)) \ge |S| + c - 1 \qquad \text{for every } x \in \Th.
\]
Consequently, if $c = |S|+1$ then $S \subseteq U$ and $G$ is not \RP{}; and if $c \ge |S|+2$ then $\Th = \emptyset$ and $G$ is not \RN{}.
\end{lemma}

\begin{proof}
We prove the inequality for every $x \in P$, first at the vertices $\iv_T$ and then by convexity.
Let $T$ be a spanning tree of $G$ and let $F$ be its edges meeting $S$.
Let $G'$ be the multigraph obtained by contracting the components of $G-S$, and $T'$ the image of $T$.
Each component becomes a single vertex, so $G'$ has exactly $|S|+c$ vertices.
Edges of $T$ with both endpoints in one component of $G-S$ become loops and are discarded.
Every other edge of $T$ meets $S$, so $E(T') = F$.
Since $T'$ is a connected spanning subgraph of $G'$, it has at least $|S|+c-1$ edges, so $|F| \ge |S|+c-1$.

We now pass from trees to the polytope.
Summing the degrees $\iv_T(E(v))$ over $S$ counts an edge of $T$ with both endpoints in $S$ twice and one with a single endpoint in $S$ once, so
\[
  \sum_{v \in S} \iv_T(E(v)) \ge |F| \ge |S|+c-1.
\]
The inequality is linear and $P$ is the convex hull of the $\iv_T$, so it holds on $P$ and hence on $\Th\subseteq P$.

For the two consequences, let $x \in \Th$ be arbitrary.
The degree inequalities bound the same sum from above, so
\[
  |S|+c-1 \le \sum_{v \in S} x(E(v)) \le 2|S|.
\]
If $c=|S|+1$ the two bounds agree, and each of the $|S|$ terms is at most $2$, so $x(E(v))=2$ for every $v \in S$.
Since $x$ was arbitrary, $S \subseteq U$; the degree inequality at a vertex of the nonempty set $S$ is tight at every point of $\Th \supseteq P^{\circ} \cap 2\FM^{\circ}$, so $G$ is not \RP{} by \Cref{lem:polytope}.

If instead $c\ge |S|+2$ the lower bound exceeds the upper bound, a contradiction.
Hence, $\Th=\emptyset$, so $P^{\circ} \cap 2\FM = \emptyset$ and $G$ is not \RN{} by \Cref{lem:polytope}.
\end{proof}

\begin{lemma}[{\cite[Proposition~3.5]{devriendt2026}}]
\label{lem:bip}
Let $G$ be connected and bipartite with parts $A$ and $B$, where $|A| \le |B|$.
If $|B| = |A|+1$ then $A \subseteq U$, and in particular $G$ is not \RP{}; if $|B| > |A|+1$ then $G$ is not \RN{}.
In either case, the toughness satisfies $\tau \le |A|/|B| < 1$.
\end{lemma}

\begin{proof}
Every edge of $G$ has an endpoint in $A$, so $G-A$ consists of the $|B|$ isolated vertices of $B$ and $c(G-A)=|B|$.
In both cases below $|B| > |A| \ge 1$, so $c(G-A) \ge 2$ and \Cref{lem:cutset} applies at $S=A$ with $c=|B|$.
If $|B|=|A|+1$ then $c=|S|+1$, and \Cref{lem:cutset} gives $A \subseteq U$, so $G$ is not \RP{}.
If $|B|>|A|+1$ then $c \ge |S|+2$, and \Cref{lem:cutset} gives $\Th=\emptyset$, so $G$ is not \RN{}.
In either case, $A$ is a cut-set with $c(G-A)=|B|>|A|$, so $\tau \le |A|/|B| < 1$.
\end{proof}

\section{Slackless vertices}
\label{sec:slack}

Under the hypothesis that $G$ is \RN{} and $2$-connected, \Cref{thm:extremal} sharpens \Cref{lem:cutset}.
A cut-set $S$ with $c(G-S)=|S|+1$ is a part of a bipartition of $G$, and the components of $G-S$ are isolated vertices.
\Cref{thm:extremal} cannot be applied at $S=U$ directly, because nothing yet gives $c(G-U)=|U|+1$.
Instead, the dual certificate of \Cref{lem:cert} lets \Cref{lem:sat} obtain the bipartite structure from $U$ itself, and $c(G-U)=|U|+1$ follows as a consequence.

\begin{theorem}
\label{thm:extremal}
Let $G$ be $2$-connected and \RN{}, and let $S \subseteq V$ with $c = c(G-S) \ge 2$.
Then $c \le |S|+1$, and in the extremal case $c = |S|+1$, both $S$ and $V \setminus S$ are independent, every component of $G-S$ is a single vertex, and $n = 2|S|+1$.
\end{theorem}

\begin{proof}
Let $C_1,\ldots,C_c$ be the vertex sets of the components of $G-S$.
Since $G$ is \RN{}, \Cref{lem:polytope} gives $\Th \supseteq P^{\circ} \cap 2\FM \neq \emptyset$.
Hence, $c \le |S|+1$ by \Cref{lem:cutset}, and it remains to treat the extremal case, so suppose $c = |S|+1$.
Fix any $x\in \Th$ and split the count equality of \Cref{lem:tree-desc} according to the components:
\begin{equation}\label{eq:parts}
n-1=x(E[S]) +\sum_{C_i}x(E[C_i])+ \sum_{C_i}x(E(S,C_i)).
\end{equation}
If any $C_i$ is a singleton, no rank inequality applies and the term $x(E[C_i])$ contributes nothing to the sum, and likewise contributes $|C_i|-1=0$ to $n-|S|-c$.
Bounding all other $x(E[C_i])$ by the rank inequality of \Cref{lem:tree-desc} gives
\begin{equation}\label{eq:split}
\begin{split}
n-1&\le x(E[S]) +\sum_{|C_i|\ge 2}(|C_i|-1) + \sum_{C_i}x(E(S,C_i))\\
&= x(E[S])+n-|S|-c+\sum_{C_i}x(E(S,C_i)).
\end{split}
\end{equation}

Summing $x(E(v))$ over $S$ counts an edge of $E[S]$ twice and an edge with one endpoint in $S$ once, so
\[
  \sum_{v\in S}x(E(v)) = 2x(E[S]) + \sum_{C_i}x(E(S,C_i)).
\]

Substituting into \Cref{eq:split} and bounding $\sum_{v\in S}x(E(v))$ by the degree inequalities of \Cref{eq:2FM} gives
\begin{equation}\label{eq:chain}
\begin{split}
n-1&\le x(E[S])+n-|S|-c+\sum_{v\in S}x(E(v)) - 2x(E[S])\\
&\le n-|S|-c+2|S|-x(E[S])\\
&=n+|S|-c-x(E[S]),
\end{split}
\end{equation}
which rearranges to
\begin{equation}\label{eq:end}
c\le |S|+1-x(E[S]).
\end{equation}
Since $c = |S|+1$, \Cref{eq:end} gives $x(E[S]) \le 0$, and each $x(e) \ge 0$ by \Cref{lem:tree-desc}, so $x(E[S]) = 0$.
All three of \Cref{eq:split,eq:chain,eq:end} are therefore equalities at every $x \in \Th$.
The edge inequality at every edge of $E[S]$ then holds with equality on all of $\Th$, and equality in \Cref{eq:split} holds term by term, so $x(E[C_i]) = |C_i|-1$ on all of $\Th$ for every $C_i$ with $|C_i| \ge 2$.
Each $|C_i| \le n-1$ because $c \ge 2$, so these are rank inequalities, and each is an implicit equality.
\Cref{thm:implicit-rn} rules that out for a $2$-connected \RN{} graph, so $E[S] = \emptyset$, the set $S$ is independent, and no $C_i$ has $|C_i| \ge 2$.
Each $C_i$ is therefore a single vertex, so $G-S$ consists of $|S|+1$ isolated vertices, $V \setminus S$ is independent, and $n = |S| + (|S|+1) = 2|S|+1$.
\end{proof}

Slacklessness is a property of the entire polytope: the degree inequality at each $v\in U$ is tight at every point of $\Th$.
Linear programming duality converts this into an identity among the constraint functions of the defining inequalities of $\Th$.
We use the following affine form of Farkas' lemma.

\begin{lemma}[Affine Farkas {\cite[Corollary~7.1h]{schrijver1986}}]
\label{lem:farkas}
Let $Q = \{x \in \R^d : g_i(x) \ge 0\ (i \in I),\ h(x) = 0\}$ be nonempty, with $I$ finite and all $g_i, h$ affine, and let $f$ be affine with $f \ge 0$ on $Q$.
Then there are $\lambda_i \ge 0$, $\mu \in \R$, and $\eta \ge 0$ with
\[
  f = \sum_i \lambda_i g_i + \mu h + \eta
\]
identically on $\R^d$.
\end{lemma}

\begin{lemma}
\label{lem:cert}
Let $G$ be $2$-connected and \RN{} with $U\neq \emptyset$.
Then there are weights $a_v \ge 1$ at each slackless vertex $v \in U$ and one constant $\mu \in \R$ such that every edge of $G$ collects total weight exactly $\mu$ from its slackless endpoints,
\begin{equation}\label{eq:star}
  \sum_{v \in e \cap U} a_v = \mu \qquad \text{for every edge } e \in E,
\end{equation}
and such that the weights are normalised against the size of $G$ by
\begin{equation}\label{eq:starstar}
  \mu (n-1) = 2\sum_{v \in U} a_v.
\end{equation}
\end{lemma}

\begin{proof}
The graph $G$ is \RN{}, so $\Th \supseteq P^{\circ} \cap 2\FM$ is nonempty by \Cref{lem:polytope}.
Fix the objective function $f(x)=\sum_{v\in U}\left(x(E(v))-2\right)$.
Each term vanishes on $\Th$ by definition of $U$, so $f$ is identically $0$ on that domain.
In particular $f \ge 0$ on $\Th$, so \Cref{lem:farkas} gives coefficients $\lambda_e, \lambda_v, \lambda_S \ge 0$ for the edge, degree, and rank inequalities, respectively, a constant $\mu \in \R$ for the count equality, and a constant $\eta \ge 0$, such that
\begin{equation}\label{eq:decomp}
\begin{split}
  f(x) = \sum_{e \in E} \lambda_e x(e)
    &+ \sum_{v \in V} \lambda_v \left(2 - x(E(v))\right) \\
    &+ \sum_{2 \le |S| \le n-1} \lambda_S \left(|S| - 1 - x(E[S])\right) \\
    &+ \mu \left(x(E) - (n-1)\right) + \eta
\end{split}
\end{equation}
identically on $\R^{E}$.

Fix $x \in \Th$.
Then $f(x) = 0$ and the count equality holds at $x$, so \Cref{eq:decomp} equates $0$ with $\eta$ plus the terms $\lambda_i g_i(x)$ over the constraints $g_i$.
Each $\lambda_i g_i(x)$ is nonnegative, as is $\eta$, and a sum of nonnegative terms vanishes only if every term does, so $\eta = 0$ and $\lambda_i g_i(x) = 0$ for every $i$.
Since $x \in \Th$ was arbitrary, $\lambda_i = 0$ unless $g_i$ is an implicit equality.
\Cref{thm:implicit-rn} shows that no edge or rank inequality is an implicit equality for a $2$-connected \RN{} graph, so $\lambda_e = 0$ and $\lambda_S = 0$ for every $e$ and every $S$.
The degree inequality at $v$ is an implicit equality exactly when $v\in U$, so $\lambda_v = 0$ when $v \in V \setminus U$.
Collecting the remaining terms,
\begin{equation}\label{eq:balance}
  \sum_{v \in U} (1 + \lambda_v) \left(2 - x(E(v))\right) = \mu \left((n-1) - x(E)\right)
\end{equation}
identically on $\R^{E}$.
Set $a_v = 1 + \lambda_v$ for $v \in U$; these satisfy $a_v \ge 1$ because $\lambda_v \ge 0$.
Comparing the coefficient of $x(e)$ on the two sides gives \Cref{eq:star}, and comparing constant terms gives \Cref{eq:starstar}.
\end{proof}

\begin{lemma}
\label{lem:sat}
Let $G$ be $2$-connected and \RN{}, with slackless set $U$.
If $U \neq \emptyset$ then $U$ is an independent vertex cover of $G$.
Consequently $n\ge 5$, and $G$ is bipartite with parts $U$ and $V \setminus U$ of sizes $|U|=(n-1)/2$ and $|V\setminus U|=(n+1)/2$, respectively.
\end{lemma}

\begin{proof}
Take $a_v \ge 1$ and $\mu$ as in \Cref{lem:cert}.
Fix $u \in U$ and an edge $e$ incident to it, which exists because $G$ is $2$-connected.
By \Cref{eq:star}, $\mu = \sum_{v \in e \cap U} a_v \ge a_u \ge 1$.
Every edge $e$ has an endpoint in $U$; otherwise, $\sum_{v \in e \cap U} a_v = 0 = \mu$, contradicting $\mu \ge 1$.
Thus, $U$ is a vertex cover and $V \setminus U$ is independent.

Next, we show that $U$ is independent.
Let $U'$ be the vertices of $U$ with a neighbour in $U$.
Suppose some $u \in U'$ has a neighbour $w \in V \setminus U$, and let $u'$ be a neighbour of $u$ in $U$.
Applying \Cref{eq:star} to $uu'$ and $uw$ and taking the difference, $0 = (a_u + a_{u'}) - a_u = a_{u'}$, contradicting $a_{u'} \ge 1$.
So, every neighbour of a vertex of $U'$ lies in $U$, and such a neighbour has a neighbour in $U'$, so it lies in $U'$ itself.
Hence, $U'$ is closed under adjacency and is a union of components of $G$.
Because $G$ is \RN{}, \Cref{lem:polytope} gives some $x \in P^{\circ} \cap 2\FM \subseteq \Th$.
If $V \setminus U$ were empty then $x(E(v)) = 2$ at every vertex, and summing over $V$ counts each edge twice, giving $2x(E) = 2n$ against the count equality $2x(E) = 2(n-1)$.
So $V \setminus U$ is nonempty.
Since $G$ is connected, its only component is $V$, which is not contained in $U' \subseteq U$.
Therefore, $U' = \emptyset$.

Every edge meets exactly one vertex of $U$, so \Cref{eq:star} gives $a_v = \mu$ for every $v \in U$.
Substituting into \Cref{eq:starstar} gives $\mu(n-1) = 2|U|\mu$.
Since $\mu\ge 1$, we divide to obtain $|U| = (n-1)/2$ and $|V\setminus U| = (n+1)/2$.
No bipartite graph on fewer than four vertices is $2$-connected, so $n \ge 4$; since $n = 2|U|+1$ is odd, this forces $n \ge 5$ and $|U| \ge 2$.
\end{proof}

\begin{remark}
Neither the hypothesis that $G$ be $2$-connected nor the hypothesis that $G$ be \RN{} can be dropped from \Cref{lem:sat}.
The path $P_4$ is not $2$-connected, and it is \RN{} because its only spanning tree is itself, so $\Th = \{\iv_E\}$ and every vertex has degree at most $2$.
Its slackless set consists of its two adjacent centre vertices.

Take now the $2$-connected graph $K_{2,3}+uw$, with the additional edge connecting the two vertices $u,w$ in the first part of the bipartition.
The polytope $\Th$ is nonempty because it contains the incidence vectors of the Hamiltonian paths of $K_{2,3}$.
By the count equality, $x(E(u))+x(E(w))=x(E)+x(uw)=4+x(uw)$.
The degree inequalities give $4+x(uw)\le 4$, so with nonnegativity $x(uw)=0$ and $x(E(u))=x(E(w))=2$ for every $x\in\Th$; that is, $u,w$ are slackless.
Thus, $K_{2,3}+uw$ is not \RN{}; otherwise, \Cref{thm:implicit-rn} gives some $x\in\Th$ with $x(uw)>0$, a contradiction.
Writing $a,b,c$ for the other part, the Hamiltonian paths $a\,u\,b\,w\,c$ and $b\,u\,a\,w\,c$ give each of $a,b,c$ degree $1$ at some point of $\Th$, so $U=\{u,w\}$.
This set is indeed a vertex cover but is not independent.
\end{remark}

\section{Consequences}
\label{sec:cons}

In this section, we derive three theorems from \Cref{thm:extremal} and \Cref{lem:sat}.
In particular, \Cref{thm:char} classifies all $2$-connected strictly \RN{} graphs as bipartite \RN{} graphs with parts whose sizes differ by one.
Alongside the paths $P_n$, this describes all connected strictly \RN{} graphs.
\Cref{thm:band} then shows that all $1$-tough \RN{} graphs are \RP{}.
Finally, \Cref{thm:tough} quantifies the toughness that an \RN{} graph may attain below $1$ as $k/(k+1)$ for $k\ge 1$, and shows that every such value is attained.

\thmchar*

\begin{proof}
First, suppose (a).
Then, since $G$ is \RN{} but not \RP{}, \Cref{thm:implicit-rp} gives $U \neq \emptyset$ and hence (b).
Next, suppose (b).
Since $U\neq \emptyset$, \Cref{lem:sat} gives that $G$ is bipartite with parts $U$ of size $k\ge 2$ and $V\setminus U$ of size $k+1$, so (c) holds.
Lastly, suppose (c).
The parts differ in size by one, so \Cref{lem:bip} applies at the smaller part and gives that $G$ is not \RP{}; hence, (a) holds.

It remains to compute $\tau$.
By (b) and \Cref{lem:sat}, $U$ is the part of size $k$.
The set $U$ satisfies $c(G-U)=|V\setminus U|=k+1$, so $\tau \le k/(k+1)$.
Let $S$ be a cut-set attaining $\tau$, so that $c(G-S)=|S|/\tau$.
Because $\tau < 1$, we have $c(G-S)>|S|$, so $c(G-S)\ge |S|+1$.
By \Cref{thm:extremal}, the reverse inequality holds, so $S$ is in the extremal case $c(G-S)=|S|+1$ and both $S$ and $V\setminus S$ are independent.
\Cref{thm:extremal} additionally yields that $n=2|S|+1$, so $|S|=k$ and $\tau=k/(k+1)$ exactly.
A connected bipartite graph has a unique bipartition up to swapping the parts, since the part containing a vertex is determined by the parity of its distance from any fixed vertex.
The bipartition $U$, $V \setminus U$ therefore coincides with $S$, $V \setminus S$, and $|S| = k = |U|$ forces $S = U$, so $U$ is the only cut-set attaining $\tau$.
\end{proof}

\thmband*

\begin{proof}
The forward direction is due to Fiedler~\cite[Theorem~3.4.18]{fiedler2011}, as applied by Devriendt~\cite[Theorem~1.3]{devriendt2026}.
For an alternate argument, suppose $G$ is \RP{} and not $1$-tough.
Then there exists a cut-set $S$ with $c(G-S)\ge |S|+1$.
\Cref{lem:cutset} then implies $G$ is not \RP{}, a contradiction, so $G$ is $1$-tough.

For the reverse, suppose $G$ is \RN{} and $1$-tough.
Every $1$-tough graph on $n \ge 3$ vertices is $2$-connected by~\cite[Proposition~1.3]{chvatal1973}, so \Cref{thm:char} applies.
Suppose $G$ were not \RP{}.
Then \Cref{thm:char} gives $\tau \le (n-1)/(n+1) < 1$, contradicting $1$-toughness.
Therefore, $G$ is \RP{}.
\end{proof}

\thmtough*

\begin{proof}
The only \RN{} graphs that are not $2$-connected are the paths $P_n$ with $n \ge 3$~\cite[Proposition~3.7]{devriendt2026}, which have $\tau = 1/2 = k/(k+1)$ for $k = 1$.
Now let $G$ be $2$-connected with $\tau < 1$, so that $G$ is not $1$-tough.
By \Cref{thm:band}, $G$ is not \RP{} and hence strictly \RN{}, so \Cref{thm:char} gives $n = 2k+1$ and $\tau = k/(k+1)$ for some $k \ge 2$.

All that remains is to show that the $2$-connected graph $K_{k,k+1}$ is \RN{} for $k\ge 2$.
Let $x=\frac{1}{|\cT|}\sum_{T\in \cT}\iv_T$ be the average of the spanning-tree indicators.
Since $x$ is a convex combination of every vertex of $P$ with every weight positive, $x\in P^{\circ}$ by \Cref{lem:tree-interior}.

Every automorphism $\sigma$ of $K_{k,k+1}$ permutes $\cT$, so $\sigma x = x$.
Let $A$ and $B$ be the parts of sizes $k$ and $k+1$.
Every vertex of $A$ is adjacent to every vertex of $B$, so any permutation of $V$ fixing $A$ setwise is an automorphism.
Given edges $ab$ and $a'b'$ with $a,a' \in A$ and $b,b' \in B$, the permutation exchanging $a$ with $a'$ and $b$ with $b'$ is therefore an automorphism sending $ab$ to $a'b'$.
Under the induced action on $\R^E$, therefore, $x(ab) = x(a'b')$, and $x$ is constant on $E$.

The count equality $x(E)=n-1=2k$ gives $x(e) = 2k/(k(k+1)) = 2/(k+1)$ for every edge $e$.
The maximum degree of any vertex in $K_{k,k+1}$ is $k+1$, achieved on any vertex of the smaller part, so $2-x(E(v))\ge 2-(k+1)\cdot 2/(k+1)= 0$.
Therefore, $x\in 2\FM$.
Thus, $x \in P^{\circ} \cap 2\FM$, so $K_{k,k+1}$ is \RN{} by \Cref{lem:polytope}.
It is bipartite and \RN{}, so \Cref{thm:char-bip} applies and $K_{k,k+1}$ is strictly \RN{} with toughness $k/(k+1)$.
\end{proof}

\section{Open questions}
\label{sec:open}

\Cref{thm:char} describes the strictly \RN{} graphs as the $2$-connected bipartite \RN{} graphs with parts of sizes $k$ and $k+1$ for some $k \ge 2$, and so reduces their classification to deciding which such bipartite graphs are \RN{}.
Expanding the definitions turns that into a question about the smaller part.
Let $G$ be connected and bipartite with parts $A$ of size $k \ge 1$ and $B$ of size $k+1$.
Every edge of $G$ has exactly one endpoint in each part, so each of the two sums below counts every edge once and every $x \in P$ satisfies
\[
  \sum_{v \in A} x(E(v)) = \sum_{v \in B} x(E(v)) = x(E) = n-1 = 2k.
\]
If $x \in P$ has $x(E(v)) \le 2$ for every $v \in A$, then the $k$ terms of the first sum are each at most $2$ and total $2k$, so $x(E(v)) = 2$ for every $v \in A$.
That is, the degree inequalities on the smaller part are tight as soon as they hold at all.
By \Cref{lem:polytope}, $G$ is \RN{} exactly when some $x \in P^{\circ}$ satisfies $x(E(v)) \le 2$ at every vertex.
With the tightness above, this says that $G$ is \RN{} if and only if some $x \in P^{\circ}$ has $x(E(v)) = 2$ for every $v \in A$ and $x(E(v)) \le 2$ for every $v \in B$.
The paths $P_{2k+1}$ and the graphs $K_{k,k+1}$ qualify; the subdivided star obtained by identifying one endpoint of each of three paths $P_3$, with parts of sizes $3$ and $4$, does not.

\begin{question}
\label{q:bip}
Which connected bipartite graphs with parts of sizes $k$ and $k+1$, for some $k \ge 1$, are \RN{}?
\end{question}

Independently, Guo, Sun, and Yang~\cite{guosunyang2026} study the same polytope $\Th$, characterising its vertices by full-rank systems of tight constraints~\cite[Theorem~3.2]{guosunyang2026} and showing that \RN{}, \RP{}, and strict \RN{} are each recognisable in polynomial time~\cite[Theorem~2.1]{guosunyang2026}.
With \Cref{thm:char}, the graphs of \Cref{q:bip} are therefore recognisable in polynomial time.
We ask instead for a structural description of them.

By \Cref{thm:band}, the two toughness thresholds of \Cref{sec:intro} coincide but remain undetermined; the bounds $3/2\le t^*_{\RN{}}\le t^*_{\RP{}}\le 5$ recorded there are the best known.

\begin{question}
\label{q:threshold}
Determine $t^* = t^*_{\RN{}} = t^*_{\RP{}}$.
\end{question}

\section*{Acknowledgements}
During the completion of this work, the author was partially supported by NSF RTG Grant DMS-2231492.
The author used Anthropic's Claude Opus 5 in developing and proofreading this work.
In response to the author's question of whether every $1$-tough \RN{} graph is \RP{} (answered by \Cref{thm:band}), Claude suggested preliminary forms of \Cref{thm:extremal} and \Cref{lem:sat} and outlined a proof strategy via LP duality.
The author wrote all proofs in this paper, correcting and completing that outline, and takes full responsibility for the content of this paper.
The author thanks Zhiyu Wang, Christin Bibby, Shea Vela-Vick, and James Oxley for helpful discussions and proofreading.
Thanks also to Gyaneshwar Agrahari, Sean Boros, and
Fernando Heidercheidt for their contributions to previous work.

\bibliographystyle{alpha}
\bibliography{references}

\end{document}